\documentclass [11pt,reqno]{amsart}
\usepackage{amsmath,amssymb,verbatim,geometry}
\usepackage[all]{xy}
\newif\ifpdf
\ifpdf
\usepackage[pdftex]{hyperref}
\else

\usepackage{xcolor}

\usepackage[hidelinks]{hyperref}
\usepackage{enumitem}
\allowdisplaybreaks

\newcommand{\Monge}{Monge-Amp\`{e}re }
\newcommand{\N}{\mathbb{N}}

\newcommand{\R}{\mathbb{R}}
\newcommand{\C}{\mathbb{C}}
\newcommand{\curl}{\mathrm{curl}}
\newcommand{\del}{\partial}
\newcommand{\Id}{\mathrm{Id}}
\newcommand{\herm}{\mathrm{herm}\,}
\newcommand{\tr}{\mathrm{tr}}
\newcommand{\kerk}{\mathfrak{k}}
\newcommand{\cchermp}[2]{ \frac{1}{2} \del #1 \otimes \bar{\del} #1 + \herm \del #2}
\newcommand{\cchermm}[2]{ -\frac{1}{2} \del #1 \otimes \bar{\del} #1 - \herm \del #2}

\newtheorem{Thm}{Theorem}[section]
\newtheorem{Prop}[Thm]{Proposition}
\newtheorem{Lem}[Thm]{Lemma}
\newtheorem{Cor}[Thm]{Corollary}

\newtheorem{Def}[Thm]{Definition}
\newtheorem{Rem}[Thm]{Remark}
\newtheorem{Ques}[Thm]{Question}

\numberwithin{equation}{section}
\title[Very weak solutions]{Very weak solutions of quadratic Hessian equations}

\date{\today}

\author{S\l awomir Dinew, Szymon Myga}

\address{Institute of Mathematics, Jagiellonian University, ul. Lojasiewicza 6, 30-348 Krak\'ow, Poland}
\email{Slawomir.Dinew@im.uj.edu.pl }
\address{Institute of Mathematics, Gda\'nsk University, ul. Wita Stwosza 57, 80-308 Gda\'nsk, Poland}
\email{szymon.myga@ug.edu.pl}
\subjclass[2010]{Primary: , secondary: }

\begin{document}
	\begin{abstract}
We extend the methods of Lewicka - Pakzad, Sz\'ekelyhidi - Cao and Li - Qiu to study the notion of very weak solutions to the complex $\sigma_2$ equation in domains in $\mathbb C^n,\ n\geq 2$. As a by-product we sharpen the regularity threshold of the counterexamples obtained by Li and Qiu in the real case.
	\end{abstract}
	\maketitle
 \section{Introduction} 
 
 In a seminal paper~\cite{LewPak17} Lewicka and Pakzad exploited the methods of convex integration to produce highly oscillating distributional solutions (dubbed very weak solutions) to the Monge-Amp\`ere equation in dimension 2 - see the next section for all the relevant definitions. In particular they were able to construct $C^{1,\alpha}$-regular solutions for $\alpha<\frac17$. Their approach was later improved by Cao and Sz\'ekelyhidi~\cite{CaoSzek19}, who used a somewhat different iteration scheme and were able to construct very weak oscillating solutions with $C^{1,\alpha},\ \alpha<\frac15$ regularity. Their result was based on the work of De Lellis, Inauen and Sz\'ekelyhidi~\cite{DIS18} concering the $C^{1,\alpha}$ immersions of disks into $(\R^3,g)$. Since then the threshold of flexibility has been moved to $\frac13$ by the authors of~\cite{CHI23}.
 
 In all those approaches the quadratic nonlinearity of the Monge-Amp\`ere operator in dimension 2 plays a crucial role. For $n\geq 3$ there seems to be no way to define a very weak Monge-Amp\`ere operator without using any Hessian terms - see \cite{Iwan01}. As a side remark we wish to point out that similar difficulty occurs when studying the domain of definition of the complex Monge-Amp\`ere operator when acting on plurisubharmonic functions - see \cite{Blo04,Blo06}.
 
 The scheme constructed in \cite{LewPak17} was later developed in a few different directions. It was apapted to the Dirichlet problem for \Monge{} equation on domains in $\R^2$ by Cao~\cite{Cao23}. On the other hand Lewicka generalized it along a geometric path to the so-called \Monge{}system, where it arises as the linearization of the curvature of deformation of a flat submanifold of any codimension~\cite{Lew25, Lew25a, Lew25b}. She further improved the regularity results for this system together with Inauen in~\cite{InaLew25, InaLew25a}.
 
 The scheme was also implemented to the more general $\sigma_2$ equations in $\R^n$ by~\cite{LiQiu}. In particular, Li and Qiu were able to produce very weak solutions of regularity $C^{1,\alpha}$ for $\alpha<\frac1{1+n+n^2}$.
 
 In the present note we extend this circle of ideas further and study the very weak solutions to the {\it complex} $2$-Hessian equation. Following the approach of~\cite{CaoSzek19} we obtain the following result:
 
 \begin{Thm}
 	Let $\Omega$ be a bounded simply-connected domain in $\C^n$. For any $ f \in L^p \left( \Omega \right) $ with $ p > n + \frac{1}{2}$ and any $ \beta < \frac{1}{1 + 2n} $ there are  $C^{1,\beta}$ very weak solutions to
 	$$S_2^{\C}(u)=f.$$
 	Furthermore these solutions are dense in $C^0 \left( \overline{\Omega} \right) $.
 \end{Thm}

Of course the very weak solutions above are in general very far from being $2$-subharmonic. As in more traditional approaches - see \cite{Blo96}, \cite{DinKo13}, \cite{Lu12} $2$-subharmonicity is assumed so that the equation is elliptic it is natural to ask under what additional conditions some {\it rigidity} occurs i.e. all very weak solutions are $2$-subharmonic and thus coincide with the ordinary weak solutions to the equation - see \cite{Blo96} for the theory associated to that latter class of functions.

Not surprisingly our approach is very close to the one in \cite{CaoSzek19, DIS18}. The essential difference is the diagonalisation result proven in Proposition \autoref{Prop:diagonalizacja} which differs from its counterpart (Proposition 3.1 in \cite{CaoSzek19}) - see Section 5 for more discussion on the topic. As a by-product we obtain similar diagonalisation for real $2$-Hessian equations (Proposition \autoref{Prop:diagonalizacja-R}) and hence we can improve the regularity exponent threshold for oscillating very weak solutions obtained by Li and Qiu in \cite{LiQiu} from $\alpha<\frac1{1+n+n^2}$ to $\alpha<\frac1{1+2n}$. Note that for $n=2$ this recovers the $\frac15$ threshold obtained by Cao and Szekelyhidi. Whether anything similar can be done for $S_k^{\C}$ or $S_k^{\R}$ operators for $n\geq k\geq3$ remains a very interesting open problem.

Apart for its own interest there are several reasons for our generalization.
 
 Arguably the most important one is to introduce this circle of ideas to specialists dealing with complex PDEs. In fact convex integration is not really new for complex analysts. Indeed, similar constructions were exploited to construct nontrivial {\it inner} functions - \cite{HakSib82}, \cite{Lo82}. To the best of our knowledge though, the convex integration was never used for studying complex PDEs of Hessian type. Hence we hope that our note will be a starting point for applying convex integration in such equations. 
 
 As a concrete example we have in mind we recall two important results due to Sver\'ak - \cite{Sve}, later generalized and improved by Lewicka, Mahadevan and Pakzad in~\cite{LewMahPak}:
 \begin{Thm}
	 Let $U$ be a domain in $\mathbb R^2$ and $c>0$ be a constant. Assume that $u\in W^{2,2}(U)$ solves the equation $det(D^2u)=f$ for some $f\geq c>0$ a.e. Then $u$ is $C^1$ and convex up to a sign change.
 \end{Thm}
\begin{Lem} Let $B_r(x)$ denote the ball centered at $x$ with radius $r$ in $\mathbb R^n$. Then for any positive definite matrix $A\in \mathbb R^{n\times n}_{sym}$ and a constant $c>0$  there is a $\varepsilon>0$ dependent only on $A$, $n$ and $c$ such that the following holds: if $u\in W^{2,n}(B_1(0))$ solves $det(D^2u)=f$ a.e. for some $f\geq c>0$ and 
	$$\int_{B_1(0)}||A-D^2u||^ndV\leq \varepsilon$$
	then $u$ is convex in the ball $B_{\frac12}(0)$. 
	
\end{Lem}
The lemma above can be thought of as an $\varepsilon$- regularity result providing rigidity (in this case - convexity) to a solution suitably close to a convex quadratic function. This is a central lemma for the proof of Sver\'ak's convexity theorem above. As the proof of the lemma heavily relies on degree theory it seems hard to prove analogous results for other operators. On the other hand Lemma 6.10 in ~\cite{LewMahPak} provides certain regularity of $u$ around the center of such a ball in real dimension $2$. Hence it is interesting to understand the full strength of the above results.  As an example Faraco and Zhong - \cite{FarZho03} have asked if a $W^{2,\infty}$ function has to be $k$-convex if $\sigma_k(D^2u)\geq\varepsilon>0$ and we know that the set of points where the Hessian matrix is $k$-admissible is of positive measure. In a different direction the following question seems very natural:

\begin{Ques}Let $B_1(0)$ be the unit ball in $\mathbb C^n$. Let $u\in W^{2,p}, p\geq n$ solves 
	$$det\left( \frac{\partial^2 u}{\partial z_j\partial\bar{z}_k} \right) =f\geq c>0.$$ 
	Fix a positive definite Hermitian matrix $A\in \mathbb C^{n\times n}_{herm}$ and assume that
	$$\int_{B_1(0)} \left\| A- \left( \frac{\partial^2 u}{\partial z_j\partial\bar{z}_k} \right) \right\|^pdV\leq \varepsilon.$$
	Is there a $p$ and $\varepsilon_0$ dependent on $p$, $n$ and $A$ so that any $\varepsilon\leq\varepsilon_0$ implies that
	$u$ is plurisubharmonic in $B_{\frac12}(0)$ in the above scenario?
	\end{Ques} 
Note that while the functions constructed by Li and Qiu in \cite{LiQiu} are quite far in terms of regularity to the reasonable thresholds in the above questions they obviously show the importance of {\it some} regularity assumptions.

Yet another reason for studying very weak solutions stems from attempts to apply implicit function theorem type arguments. In the case of Hessian type equations the admissiblity requirement is a well-known obstruction for applying such arguments when working over Sobolev spaces. Typically this is circumvented through working in H\"older spaces instead (of sufficiently high regularity) but such an approach has its limitations. Hence it seems that constructing a theory of solutions that generalizes past the admissibility requirement is an important step for building such a theory.

Lastly the theory of $L^p$-viscosity solutions to complex Hessian equations was recently initiated - see \cite{ADO24}. As noted in \cite{ADO24} there are still numerous problems with establishing a good notion of a {\it supersolution} in that setting - once again these are caused by the admissibility requirement. Our hope is that very weak solutions of suitable Sobolev regularity can serve as test functions within such a theory.

Our note is organized as follows: in Section 2 we collect all the necessary definitions and conventions. Section 3 contains the description of the machinery of very weak solutions together with a dictionary linking these tools with the standard notation for complex Hessian equations. In Section 4 we provide the essential toolkit needed to apply convex integration. Section 5 deals with the crucial diagonalisation result on which the whole argument hinges. In Section 6 we repeat the arguments from \cite{CaoSzek19} and~\cite{LiQiu} in order to obtain oscillating very weak solutions through an iteration procedure. Finally in Section 7 we briefly explain how to apply analogous reasoning in the real case.

{\bf Acknowledgements}

 The first named author was partially supported by grant
no. 2021/41/B/ST1/01632 from the National Science Center, Poland. Part of the work was done while the second named author was visiting Jagiellonian Univeristy and he wishes to express his gratitude for the hospitality. Both authors wish to thank the anonymous referee for his/her remarks.

\section{Preliminaries}
Below we define all prerequisites for the remaining part of the note. We begin with settling the notation:

{\bf Notation.} We follow the notation form \cite{GiTr}.

 By $|\cdot|_k$ we define the semi-norm
\[
|u|_k := \max_{|\alpha|= k } \sup_\Omega \left| \frac{\del^\alpha u}{\del  x_1^{\alpha_1} \ldots \del x_k^{\alpha_k}}(x)\right|.
\]
For $k = 0$ this is simply a supremum of $|u|$ and is in fact a norm. By $\| \cdot \|_k$ we define the norm
\[
\| u \|_k := \sum_{j = 0}^k |u|_j.
\]
H\"{o}lder norms and semi-norms will be denoted by Greek letters and are defined as
\[
| u |_\alpha := \sup_{x \neq y} \frac{| u(x) - u(y) | }{|x - y|^\alpha} 
\]
and
\[
\| u \|_{k;\alpha} := \| u \|_k + \sum_{|\alpha| = k} \left| \frac{\del^\alpha u}{\del  x_1^{\alpha_1} \ldots \del x_k^{\alpha_k}} \right|_\alpha.
\]

$\R^{n\times n}_{sym}$, will be the set of all symmetric real $n\times n$ matrices, while $\R^{n\times n}_{sym}(U)$ will be the set of matrix valued functions defined on a domain $U$ and with values in $\R^{n\times n}_{sym}$. Analogously  $\C^{n\times n}_{herm}$ and $\C^{n\times n}_{herm}(\Omega)$ denote the Hermitian symmetric matrices and Hermitian symmetric valued functions, respectively.

As customary by $C$ we shall denote a generic constant that depends only on the pertinent quantities. In particular it may change line to line. If confusion may arise additional dependencies will be hinted. 

The operators $\wedge$ and $\otimes$ denote the standard exterior and tensor product. By $*$ we shall denote the Hodge operator both in the real and complex setting (in the complex setting we use the convention $\int\alpha\wedge\overline{*\beta}=\int<\alpha,\beta>$ for forms of the same bidegree). It will be clear from the context which one will be used.

The $\curl$ operator acts on a real $2$-vector field $v=(v^1,v^2)$ on (a subset of) $\R^2$ by 
\[\curl (v):= \frac{\partial v^1}{\partial y}(x,y)-\frac{\partial v^2}{\partial x}(x,y).\]

The definition extends from vector fields to tensors in a natural way. More general $\curl$-type operators will be introduced later on.

We shall use the standard complex notation whe dealing with complex PDEs. In particular for $z=(z_1,\cdots,z_n)$ we have $z_j = x_j + iy_j, j=1,\cdots, n$ and we put
\[\mathfrak{R}(z_j)=x_j,\ \mathfrak{I}(z_j)=y_j;\]
\[
\del_{z_j} = \frac12(\del_{x_j} - i\del_{y_j}), \quad \del_{\bar{z}_j} = \bar{\del}_{z_j} =\frac12(\del_{x_j} + i \del_{y_j}).
\]
Similarly we define the differential one forms
\[
dz_k = dx_k + idy_k, \quad d\bar{z}_k = dx_k - idy_k.
\]
and build from these the spaces of $(p,q)$ forms $0\leq p\leq n,\ 0\leq q\leq n$ in the usual way.

Finally for a $n\times n$ matrix $M$ its norm $||M||$ will be defined by
\[||M|| := \sup_{x\in S^{n-1}}||Mx||.\]

 \section{Definining the very weak solutions for quadratic Hessians}

Let $\Omega$ be a simply-connected domain in $\R^2$. In~\cite{Iwan01} Iwaniec defined the very weak version of \Monge operator on $W_{loc}^{1,2}(\Omega)$ as
\[
MA(v) = \del^2_{xy}(\del_x v \del_y v) - \frac{1}{2}\del^2_{yy} (\del_x v)^2 - \frac{1}{2} \del^2_{xx} (\del_y v)^2,
\]
which can be interpreted as an operator on matrix fields
\[
MA(v) = -\frac{1}{2} \curl \curl (\nabla v \otimes \nabla v)
\]
and the first curl can be taken either row-wise or column-wise since the matrix is symmetric. The definition can be extended in two directions, first we may define the 2-Hessian operator in $\R^n$ and then we can extend it to the 2-Hessian in $\C^n$. Let us recall that the 2-Hessian operator usually denoted by $S_2(u)$ is defined as 
\begin{equation}\label{eq:2-Hess-klas}
	S_2(u) = \sigma_2 \left( D^2u \right) = \sum_{i < j} \del^2_{ii} u \, \del^2_{jj} u - (\del^2_{ij} u)^2.
\end{equation}
Here $\sigma_2 \left( D^2 u \right) $ stands for second elementary symmetric polynomial of the eigenvalues of $D^2 u$.

Before we define the 2-Hessian in $\R^n$ we must take a closer look at what it would mean to define a row-wise or column-wise curl of a matrix field. For this problem we shall use the language of differental forms. A matrix field can be looked at as a section of a tensor bundle $\left( \bigwedge^1 \R^n \otimes \bigwedge^1 \R^n \right) (\Omega)$. In that bundle we define a left-curl as
\[
\curl_L (f\, dx \otimes dy) := *d(fdx) \otimes dy.
\]
Similarly we define the right-curl ($\curl_R$). Therefore the result of $\curl_R \curl_L M$ for some matrix field $M$ will be a tensor field in $\left( \bigwedge^{n-2} \R^n \otimes \bigwedge^{n-2} \R^n \right) (\Omega)$. Let us notice that the definition is consistent with usual curls in 2 and 3 dimensions. Finally, as the outcome should be  a scalar function we need to introduce an appropriate trace operator, which in the language of differential forms comes naturally. Let $ dx_I, dx_J$ be some elements of the canonical basis of $\bigwedge^{k} \R^n$ then we define
\[
\mathrm{tr} \left( f dx_I \otimes dx_J \right) := * \left( f dx_I \wedge *dx_J \right)
\]
Let us notice that for the matrix fields the result of this operation will be indeed their trace. Now we are able to introduce the very weak form of the 2-Hessian operator:

\begin{Def}
	Let $\Omega$ be a simply-connected domain in $\R^n$ and let $u \in W^{1,2}_{loc}(\Omega)$. We say that $u$ is a very weak solution to the 2-Hessian equation
	\[
	S_2(u) = f
	\]
	if $u$ solves the equation
	\[
	-\frac{1}{2} \tr \left( \curl_L\curl_R \left( d u \otimes d u \right) \right) = f
	\]
	in the weak sense.
	
\end{Def}

\begin{Rem}
	The reader can check that if instead of integrating by parts we differentiate the coefficients of the tensor field we will indeed end up with the operator~\eqref{eq:2-Hess-klas}.
\end{Rem}

\paragraph{ \bf Comparison with $\mathfrak{C}^2$.} At this point it is worthwhile to compare the operators $\curl\,\curl$ and $\mathfrak{C}^2$. The operator $\mathfrak{C}^2$ was introduced by Lewicka in~\cite{Lew25}. For a regular enough matrix field $A \in \mathbb{R}^{n \times n}_{sym}(\Omega)$ the operator $\mathfrak{C}^2(A): \Omega \to \R^{n^4}$ is defined as
\[
\mathfrak{C}^2(A)_{ij,st} = \del_i \del_s A_{jt} + \del_j \del_t A_{is} - \del_i \del_t A_{js} - \del_j \del_s A_{it}.
\]

Let us reinterpret $\mathfrak{C}^2$ in the language of differential forms. We may think of $A$ as a section of $S^2 T^*\Omega$ with canonical basis denoted by $dx_j \odot dx_t$. Then if we compute the second derivative of $A$, we end up with a section of $S^2(T^*\Omega \otimes T^*\Omega)$ with coefficients
\[
\del_i \del_s A_{jt} \, (dx_i \otimes dx_j) \odot (dx_s \otimes dx_t).
\]
If we would like to make a differential form out of it we must alternate coefficients, thus we get
\[
\mathrm{Alt}_{1,2} \mathrm{Alt}_{3,4} \nabla^2 A \in \left( \bigwedge\nolimits^2 \Omega \right) \otimes \left( \bigwedge\nolimits^2 \Omega \right).
\]
Looking at the coefficients of a resulting object we get
\[
(\mathrm{Alt}_{1,2} \mathrm{Alt}_{3,4} \nabla^2 A)_{ij, st} = \del_i \del_s A_{jt} + \del_j \del_t A_{is} - \del_i \del_t A_{js} - \del_j \del_s A_{it}
\]
for $i < j$ and $s < t$.

On the other hand applying $d_L$ and $d_R$ defined as
\[
 d_L( f \, dx \otimes dy ) := d(f dx) \otimes dy, \qquad d_R( f \, dx \otimes dy ) := dx \otimes d( f dy) 
\]
to $A = A_{jt} dx_j \otimes dx_t$ we get
\[
(\mathrm{Alt}_{1,2} \mathrm{Alt}_{3,4} \nabla^2 A) = d_L d_R (A).
\]
It is now easy to see that $\mathfrak{C}^2$ is basically identical to $\curl_L \curl_R$ up to an antisymmetrization and the Hodge star, which implies that
\[
\ker \left( \mathfrak{C}^2 \right) \subseteq \ker \left( \tr \, \curl\curl \right).
\]

Since $\mathfrak{C}^2$ is the linearization of Riemannian curvature of shallow manifolds, our operator $\tr \,\curl\curl$ has a natural interpretation as the linearization of the \emph{scalar} curvature of shallow manifolds. It is thus natural to ask for a similar result in the case of linearized \emph{Ricci} curvature.

\paragraph{ \bf Complex case.} The complex 2-Hessian operator is defined by
\[
S_2^\C(u) = \sigma_2 \left( \left[ \frac{\del^2 u}{\del z_k \del\bar{z}_j} \right]_{k,j} \right) = \sum_{k < j} \del^2_{z_k \bar{z}_k} u \, \del^2_{z_j \bar{z}_j} u - \del^2_{z_k \bar{z}_j} u \, \del^2_{z_j \bar{z}_k} u.
\]
To finally introduce the very weak formulation we recall two more standard operators in the complex variables, namely the $(1,0)$ and $(0,1)$ parts of the exterior derivative $d$:
\[
\del u  = \sum_k \del_{z_k}u\, dz_k, \quad \bar{\del} u = \sum_j \del_{\bar{z}_j}u \, d\bar{z}_j
\]

\begin{Def}
	For  tensor fields over $\C$ we define
	$\curl_L$ as
	\[
	\curl_L (f dz \otimes d\bar{w}) := \left( *\del ( fdz ) \right) \otimes d\bar{w}
	\]
	and $\overline{\curl}_R$ as
	\[
	\overline{\curl}_R (f dz \otimes d\bar{w}) := dz \otimes \left( * \bar{\del} ( f d\bar{w} ) \right).
	\]
\end{Def}
\begin{Def}
	Let $dz$ and $dw$ be elements of the canonical basis of $\bigwedge^{k,p} \C^n$ and $\bigwedge^{p,k} \C^n$ respectively then we define the trace as 
	\[
	\tr \left( f dz \otimes dw \right) := * \left( f dz \wedge \overline{ * dw} \, \right).
	\]
\end{Def}

We may note again that after identifying complex matrices with elements of $\bigwedge^{1,0} \C^n \otimes \bigwedge^{0,1} \C^n$ the above definition reproduces the usual trace operator.

\begin{Def}
	Let $\Omega$ be a simply connected domain in $\C^n$. We say that $u$ is a very weak solution to the 2-Hessian equation
	\begin{equation}\label{eq:cx-2-Hessian}
		S^\C_2(u) = f
	\end{equation}
	if $u$ satisfies
	\[
	-\frac{1}{2}\tr \left( \curl_L\overline{\curl}_R \right) \left( \del u \otimes \bar{\del} u \right) = f.
	\]
	in the weak sense.
\end{Def}

In the complex case the 2-Hessian has a much simpler definition of very weak solutions. Let us introduce the operator
\[
d^c = \frac{i}{2}(\bar{\del} - \del).
\]
We say that $u$ is the very weak solution of the 2-Hessian equation~\eqref{eq:cx-2-Hessian} if it satisfies
\[
dd^c \left( du \wedge d^c u \right) \wedge \left( dd^c \frac{|z|^2}{2} \right)^{n-2} = f
\]
in the weak sense.

{{\bf The kernel of} $\curl\overline{\curl}$} 

Let us consider the kernel of the $\mathrm{curl}\overline{\mathrm{curl}}$  operator in the space $\C^{n \times n}_\herm \left( \Omega \right)$. As in the case of $n = 2$ it will contain matrices of the form $\herm \del w$ for any regular enough $w :\C^n \rightarrow \C^n$. Recall that
\[
\herm \del w := \frac{1}{2} \left( \del w + \left( \del w \right)^{*} \right),
\]
where $^*$ stands for the Hermitian conjugate and $\del w$ is the matrix $\left[ \del w_i / \del z_j \right]_{1 \leq i,j \leq n}$. In order to investigate the kernel when $n > 2$ we need the following lemma which follows directly from computation:
\begin{Lem}
	Let $i,j\in\lbrace1,\cdots,n\rbrace$ and $a_{i\bar{j}}$ be  smooth complex valued functions. Then
	\begin{equation}\label{aij}
		\tr \, \curl_L \overline{\curl}_R \left( a_{i\bar{j}} dz_i \otimes d\bar{z}_j \right) =\begin{cases}
			-\del_{j\bar{i}} a_{i \bar{j}} \ {\rm if}\ i\neq j;\\
			 \sum_{j \neq i} \del^2_{j\bar{j}} a_{i \bar{i}}=: \Delta_i a_{i \bar{i}}\ {\rm if}\ i=j.
		\end{cases} 
	\end{equation}
\end{Lem}
As a corollary we obtain the following description of the kernel:
\begin{Cor}
	The matrix valued function $(a_{i\bar{j}})\in \C^{n \times n}_\herm \left( \Omega \right)$ satisfies \newline
	$\tr \, \curl_L \overline{\curl}_R \left( a_{i\bar{j}} dz_i \otimes d\bar{z}_j \right)=0$
	if and only if
	\begin{equation}\label{eq:2-Hess-kernel}
		\sum_j \Delta_j a_{j\bar{j}} = 2 \sum_{i < j} \mathfrak{R} \left( \del_{i\bar{j}} a_{j\bar{i}} \right).
	\end{equation}
\end{Cor}

{\bf Results.}

We are now ready to formulate the main results in the note. Let $\Omega$ be a bounded simply-connected domain in $\C^n$. Given the $S_2^\C \left( v \right)$ operator defined above we are looking for a very weak solution to the problem
\begin{equation} \label{eq:MA-vw}
	S^\C_2(v) = f. \tag{$\star$}
\end{equation}

In that direction we prove that the very weak solutions to~\eqref{eq:MA-vw} abound provided they are in class $C^{1,\beta}$ for small enough $\beta$.

\begin{Thm}\label{Thm:main}
	Let $\Omega$ be a bounded simply-connected domain in $\C^n$. For any $ f \in L^p \left( \Omega \right) $ with $ p > n + \frac{1}{2}$ and any $ \beta < \frac{1}{1 + 2n} $ the $C^{1,\beta}$ very weak solutions to~\eqref{eq:MA-vw} are dense in $C^0 \left( \overline{\Omega} \right) $.
\end{Thm}

The proof is very similar as in the case of the real 2-Hessian announced recently in~\cite{LiQiu}. However we were able to achieve a better H\"{o}lder exponent thanks to a generalization of the kernel trick by Cao and Sz\'ekeleyhidi~\cite[Proposition 3.1]{CaoSzek19}, which in turn was motivated by the key proposition of~\cite{DIS18}. This generalization works (with slightly different proofs) for both real and complex 2-Hessian operators thus we may apply it to the analogus result in the real case to get an improvement of the result of Li and Qiu. Namely the following holds: 
\begin{Cor}
	Let $\Omega$ be a bounded simply-connected domain in $\R^n$. For any $ f \in L^p \left( \Omega \right) $ with $ p > \frac{2n + 1}{4} $ and any $ \beta < \frac{1}{1 + 2n} $ the $C^{1,\beta}$ the very weak solutions to
	\[
	S_2 (u) = f
	\]
	are dense in $C^0 \left( \overline{\Omega} \right) $.
\end{Cor}

\section{Convex integration toolkit}

First let us recall an easy, but essential interpolation inequality. There is a constant dependent on $\alpha$ such that for any $C^1$ function $f$ we have
\begin{equation}\label{eq:Holder-interpolation}
	| f |_\alpha \leq C \| f \|_0^{1 - \alpha} | f |_1^{\alpha}.
\end{equation}

Next we would like to recall the relationship between the $C^k$ norms and mollification. The following are classical estimates that can be found in \cite{GiTr} or \cite{CaoSzek19}.

\begin{Lem}\label{Lem:mollify-est}
	Let $\phi$ be a standard mollifier (i.e. smooth nonnegative function, supported in the unit ball, rotationally symmetric and integrable to 1). By $\phi_l$ for any $l >0$ we denote $l^{-n} \phi \left( \frac{x}{l} \right) $. Then if $f * \phi_l$ is a convolution we get
	\begin{align}
		\| \nabla^{ (m) } \left( f * \phi_l \right)\|_0 &\leq \frac{ C_m \| f \|_0 }{ l^m };\label{uselateron} 
		\\
		| f * \phi_l |_{k;\alpha} &\leq C(l,\alpha) | f |_k; \label{eq:est-mollified-Holder}
		\\
		\| f - f * \phi_l \|_0 &\leq C \min \left\{ l^2 \| \nabla^2 f \|_0 , l \| \nabla f \|_0, l^\beta \| f \|_{0,\beta} \right\}; \label{eq:est-mollifier-roznica} \\
		\| (fg) * \phi_l  - (f * \phi_l)(g * \phi_l) \|_\alpha &\leq C \| f \|_1 \| g \|_1 l^{2-\alpha}. \label{eq:est-mollifier-iloczyn}
	\end{align}
	
\end{Lem}

Next we recall the additional functions used to perturb  an approximate solution in the iteration scheme.
Keeping with the notation introduced in~\cite{CaoSzek19} we take for the perturbations the following functions:
\[
\Gamma_1 (s,t) = \frac{s}{\pi} \sin (2\pi t), \quad
\Gamma_2 (s,t) = -\frac{s^2}{4\pi} \sin (4\pi t)
\]
and note that they satisfy
\begin{align}\label{eq:perturbacje-rown}
	\Gamma(s, t+1) &= \Gamma(s,t), \nonumber \\
	\frac{1}{2}|\del_t\Gamma_1(s,t)|^2 + \del_t\Gamma_2(s,t) &= s^2.
\end{align}
along with estimates
\begin{align}
	\begin{split}\label{eq:perturbacje-est}
	| \del^k_t \Gamma_1 | + | \del_s \del_t^k \Gamma_2 | &\leq C_ks; 
	\\
	| \del_s \del_t^k \Gamma_1 | &\leq C_k; 
	\\
	| \del_t^k \Gamma_2 | &\leq C_ks^2
	\end{split}
\end{align}
for any nonnegative integer $k$. Here $C_k$ are numerical constants dependent only on $k\in\mathbb N$.

\section{Diagonalization}

In this section we will prove the analogue of the diagonalization proposition from~\cite[Proposition 3.1]{CaoSzek19}. In order to simplify the notation from now on we shall write $\mathrm{curl}\overline{\mathrm{curl}}$
instead of $\mathrm{curl}_L\overline{\mathrm{curl}}_R$.

\begin{Prop}\label{Prop:diagonalizacja}
	For any $j \in \N , \ 0 < \alpha < 1$ there are constants $M_1, M_2, \ldots$ and $\tilde{\sigma}$ depending on $j, \alpha$ and $n$ such that the following holds. If $H \in C^{j,\alpha} \left( \Omega, \C^{n \times n}_{\herm} \right)$ satisfies
	\[
	\| H - \Id \|_\alpha \leq \tilde{\sigma},
	\]
	then there exist $\kerk \in C^{j+1, \alpha} \left( \Omega, \C^{n \times n} \right) \, \cap \, \ker_w \left( \mathrm{curl}\overline{\mathrm{curl}} \right)$ and $d_1,\ldots,d_n \in C^{j, \alpha} \left( \Omega, \R \right)$ such that
	\begin{equation}\label{eq:diag-potencjal}
		H + \kerk = diag(d^2_1,\ldots,d^2_n)
	\end{equation}
	and the following estimates hold:
	\begin{align}
		\begin{split}\label{eq:diago-est}
			\sum_{k=1}^n \| d_k - 1 \|_\alpha + \| \kerk \|_\alpha &\leq M_1 \| H - \Id \|_\alpha \\
			\sum_{k=1}^n\| d_k \|_{j ; \alpha} + \| \kerk \|_{j ; \alpha} &\leq M_j \| H - \Id \|_{j ; \alpha}.
		\end{split}
	\end{align}
\end{Prop}
\begin{Rem}
	By $\mathrm{ker}_w$ we mean the space of weak solutions to the kernel equation~\eqref{eq:2-Hess-kernel}. The exact conditions defining this space will be written down shortly, at the begining of the proof. 
\end{Rem}

\begin{proof}
	Let us start by denoting the entries of $H$ by $h^{j\bar{k}}$. 
	
	Next, we will write down the weak formulation of the kernel equation~\eqref{eq:2-Hess-kernel}. Fix a test function $\varphi$:
	\begin{align}
		0 &= \int_\Omega \varphi \, \tr \curl\overline{\curl}\kerk 
		= 
		\int_\Omega \varphi \left( \sum_{j} \Delta_j \kerk^{jj} - \sum_{k \neq j} \del^2_{k\bar{j}}\kerk^{j\bar{k}} \right) \nonumber \\
		\begin{split}\label{eq:cx-Hessian-slabe-jadro}
			&= -\int_\Omega \sum_j \del_{j}\varphi \frac{1}{2} \left( \sum_{k \neq j } \del_{\bar{j}} \kerk^{k\bar{k}} - \del_{\bar{k}} \kerk^{k\bar{j}} \right) + \sum_j \del_{\bar{j}}\varphi \frac{1}{2} \left( \sum_{k \neq j } \del_{j} \kerk^{k\bar{k}} - \del_{k} \kerk^{j\bar{k}} \right).
		\end{split}
	\end{align}
	We wish to represent  the above integral as
	\[
	\int_\Omega  d\varphi \wedge dF
	\]
	for some $(2n-2)$-form $F$. Let us denote
	\begin{align*}
		dz^{j,k} &:= \frac{i}{2}dz_1 \wedge d\bar{z}_1 \wedge \ldots \wedge \frac{i}{2}dz_{j-1} \wedge d\bar{z}_{j-1} \wedge \frac{i}{2}dz_{j+1} \wedge d\bar{z}_{j+1} \wedge \ldots \wedge \\
		&\wedge \frac{i}{2}dz_{k-1} \wedge d\bar{z}_{k-1} \wedge \frac{i}{2}dz_{k+1} \wedge d\bar{z}_{k+1} \wedge \ldots \wedge \frac{i}{2}dz_{n} \wedge d\bar{z}_{n} \\
		dz^{j} &:= dz^{j,j} \ \ {\rm i.e.\ only\ } \frac{i}{2} d{z}_j \wedge d\bar{z}_j\ {\rm is\ omitted}.
	\end{align*}
	Let us also recall that the canonical volume form in $\C^n$ can be expressed as
	\[
	d\mathrm{vol} = \bigwedge_{j = 1}^n \frac{i}{2} dz_j \wedge d\bar{z}_j,
	\]
	from which it follows that the Hodge star operator works as
	\[
	*\left( \frac{i}{2} dz_j \wedge d\bar{z}_k \right) = -\frac{i}{2} d{z}_j \wedge d\bar{z}_k \wedge dz^{j,k}, \quad *\left( \frac{i}{2} dz_j \wedge d\bar{z}_j \right) = dz^{j} .
	\]
	In such a notation $dF$ can be expressed as
	\[
	dF = \sum_{j} F^j \frac{i}{2} dz_j \wedge dz^j + F^{\bar{j}} \frac{i}{2} d\bar{z}_j \wedge dz^{j}
	\]
	for some complex valued functions $F^j, F^{\bar{j}},\, j = 1,\ldots, n$. Comparing with~\eqref{eq:cx-Hessian-slabe-jadro} we arrive at the following identities
	\begin{align}
		\begin{split}\label{eq:dF-kernel}
			F^{\bar{j}} &=\frac12\sum_{k \neq j } \del_{\bar{j}} \kerk^{k\bar{k}} - \del_{\bar{k}} \kerk^{k\bar{j}}; \\
			F^{j} &=\frac12 \sum_{k \neq j } \del_{k} \kerk^{j\bar{k}} - \del_{j} \kerk^{k\bar{k}}.
		\end{split}
	\end{align}
	
	Now we require that $F$ to be of the form $*G$ for some $2$-form $G$. We can be even more specific and require that $G$ be a $(1,1)$-form. In that case let us denote the coefficients of $G$ by
	\[
	G = \sum_{j \neq k} \frac{i}{2} g^{j\bar{k}} \, \frac{i}{2} dz_j \wedge d\bar{z}_k + \sum_j \frac{i}{2} g^{j \bar{j}} \, \frac{i}{2} dz_j \wedge d\bar{z}_j.
	\]
	
	We can now define the coefficients of $G$ as
	\begin{align}
		\begin{split}\label{eq:potencjaly-2-formy}
			g^{k\bar{j}} &= \del_{j\bar{j}}\psi^{j\bar{k}} - \del_{k\bar{k}} \psi^{j\bar{k}} - \sum_{l \notin \{j,k \} } \del_{k\bar{l}} \psi^{j\bar{l}} + \sum_{l \notin \{j,k \} } \del_{l\bar{j}}\psi^{l\bar{k}}\ {\rm\ if}\ k\neq j, \\
			g^{j\bar{j}} &= \sum_{k \neq j} \del_{k\bar{j}}\psi^{k\bar{j}} - \sum_{k \neq j} \del_{j\bar{k}} \psi^{j\bar{k}}k,
		\end{split}
	\end{align}
	where $\psi^{k\bar{j}}$ are smooth complex valued functions to be determined later on.
	On the other hand, since $dF = d *G$ for the coefficients of $dF$ we must have
	\begin{align}
		\begin{split}\label{eq:dF=d*G}
			F^j &= \del_{j}g^{j\bar{j}} + \sum_{k \neq j} \del_{k} g^{j\bar{k}}, \\
			F^{\bar{j}} &= \del_{\bar{j}}g^{j\bar{j}} + \sum_{k \neq j} \del_{\bar{k}} g^{k\bar{j}}.
		\end{split}
	\end{align}
	Therefore combining~\eqref{eq:dF-kernel},~\eqref{eq:potencjaly-2-formy} and~\eqref{eq:dF=d*G} we arrive at the following formula
	\begin{align*}
		\frac{1}{2} \sum_{k \neq j} \del_{\bar{j}} \kerk^{k\bar{k}} - \del_{\bar{k}} \kerk^{k\bar{j}} = &\sum_{k \neq j} \del_{\bar{j}k\bar{j}} \psi^{k\bar{j}} - \del_{\bar{j}j\bar{k}}\psi^{j\bar{k}}
		\\
		+ &\sum_{k \neq j} \left( \del_{\bar{k}j\bar{j}}\psi^{j\bar{k}} - \del_{\bar{k}k\bar{k}} \psi^{j\bar{k}} - \sum_{l \notin \{j,k \} } \del_{\bar{k}k\bar{l}} \psi^{j\bar{l}} + \sum_{l \notin \{j,k \} } \del_{\bar{k}l\bar{j}}\psi^{l\bar{k}} \right)
		\\
		= & \sum_{k \neq j} -\del_{\bar{k}} \Delta \psi^{j\bar{k}} + \sum_{k \neq j} \del_{\bar{j}} \left(  \sum_{ l \neq m} \del_{l\bar{m}} \psi^{l \bar{m}} \right)
	\end{align*}
	and for $F^j$ we get 
	\begin{align*}
		\frac{1}{2} \sum_{k \neq j} \del_k \kerk^{j\bar{k}} - \del_j \kerk^{k\bar{k}} = &\sum_{k \neq j} \del_{jk\bar{j}}\psi^{k\bar{j}} - \del_{jj\bar{k}}\psi^{j\bar{k}} 
		\\
		+ &\sum_{k \neq j} \left( \del_{kk\bar{k}}\psi^{k\bar{j}} - \del_{kj\bar{j}} \psi^{k\bar{j}} - \sum_{l \notin \{j,k \} } \del_{kj\bar{l}} \psi^{k\bar{l}} + \sum_{l \notin \{j,k \} } \del_{kl\bar{k}}\psi^{l\bar{j}} \right) 
		\\
		= &\sum_{k \neq j} \del_k \Delta \psi^{k\bar{j}} - \sum_{k \neq j} \del_j \left( \sum_{l \neq m} \del_{l\bar{m}}\psi^{l\bar{m}} \right).
	\end{align*}
	
	Therefore we may define the entries of the weak kernel as 
	\[
	\kerk^{k\bar{j}} = 2 \Delta \psi^{j\bar{k}}, \quad \kerk^{j\bar{j}} = 2 \sum_{l \neq m} \del_{l\bar{m}} \psi^{l\bar{m}}
	\]
	for any collection of regular enough functions $\psi^{j\bar{k}}$. In particular we may require that they solve the zero-boundary Dirichlet problem $2 \Delta \psi^{j\bar{k}} = -h^{j\bar{k}}$ $(j\neq k)$, which then makes $\kerk$ hermitian. We get the desired inequalities for $\kerk$ from the Schauder estimates for the entries of $\psi^{j\bar{k}}$ in the following way: the off diagonal entries are small in $C^{\alpha}$ norm from the assuptions made on $H$, whereas the diagonal entries are then small due to the $C^{2,\alpha}$ smallness of $\psi^{j\bar{k}}$ and the formula above. The estimates for the diagonal entries $(d^2_1 , \ldots, d^2_n)$ easily follow from here since we may take $\tilde{\sigma}$ small enough so that $H + \kerk$ is sufficiently close to $\Id$.
\end{proof}

\section{Proof}
As the details below are somewhat technical (especially for non-specialists in convex integration) we briefly sketch the main points. First it is fairly easy to solve
\[
-\frac{1}{2} \tr \left( \curl\overline{\curl} \, A \right) = f.
\]
for an unknown {\it matrix} $A$ and the solution is very far from being unique. Thus we find a solution which is a diagonal matrix. Next the main point is to decompose $A$ into $\del v \otimes \bar\del v $ plus an element belonging to the kernel of $\tr\,\curl\overline{\curl} $ of the form $\herm w + \kerk$. This is done through an iteration procedure. First one picks a well chosen {\it approximate solution} so that each $v,w$ and $\kerk$ are under control in suitable norms and the difference, dubbed the deficit tensor, is also controlled. Then the main technical tool (see Proposition \ref{Prop:Stage}) is that one can find {\it better} approximants, with even smaller deficit tensor. Thus the limiting data of such a procedure, provided it exists, will give the desired decomposition.

From the sketch above it is clear that such an iteration scheme crucially relies on two elements: on the possibility to perform the iteration step and on the convergence in the limit. The second issue is resolved by keeping track of the smallness of norms of the differences of the data from two consecutive steps. The first one is more technical in nature: the idea is to first mollify the given data at a given step using a precisely described mollifier parameter and then to perturb the outcome through adding small-amplitude but high-frequency functions. The quadratic nonlinearity of $S_2^{\mathbb C}$ is heavily used exactly in this perturbation procedure.

We proceed with  details of the proof of Theorem \autoref{Thm:main}:

{\bf Preparation}. Let us fix $n$ to be the complex dimension of $\Omega$. For now we are looking for a matrix solution to the equation
\[
-\frac{1}{2} \tr \left( \curl\overline{\curl} \, A \right) = f.
\]
Notice that putting $A = (u + \tau)\mathrm{Id}$ for any constant $\tau$ gives 
\begin{equation}\label{AB}	
-\frac{1}{2} \tr \left( \curl\overline{\curl}\,\left[ (u + \tau)\mathrm{Id} \right] \right) = (1-n)\Delta u = f
\end{equation}
and this is always weakly solvable with $W^{2,p} ( \Omega )$ solution provided $f \in L^p\left( {\Omega} \right),\ p>1$. 

We will solve~\eqref{eq:MA-vw} if we are able to represent $A$ as
\[
A = \del v \otimes \bar\del v \ + \herm w + \kerk
\]
for some function $v: \Omega \rightarrow \R$, map $w: \Omega \rightarrow \C^n$ and some element of the kernel $\kerk : \Omega \rightarrow \C^{n \times n}_
\herm$ (of course $\herm w$ is also in the kernel of $S_2^\C$, but it serves a different purpose than $\kerk$, so it is better to write them apart). Let us recall that in order for that to be possible we must have $v \in W^{1,2}(\Omega)$ which is always satisfied whenever $\del v \otimes \bar{\del} v \in W^{2,p}$, for any $p \geq 1$. Moreover, for the proof of flexibility we require $A$ to be H\"{o}lder continuous (thus extandable to $\overline{\Omega}$), which by Morrey's theorem imposes the assumption $p > n$. Then $A \in C^{0,\kappa}$ with $\kappa = 2 - \frac{2n}{p}$, and hence $u \in C^{0,\kappa}$.

By density of $C^\infty(\overline{\Omega})$ in $C^0(\overline{\Omega})$, to prove our main theorem it is enough to approximate any smooth function with solutions to \eqref{eq:MA-vw}. Thus, let $u^{\flat}$ be any fixed smooth real-valued function defined on $\overline{\Omega}$. Recall that $u$ solves the Poisson equation \ref{AB}. Using a neat idea from~\cite{CHI23} let us fix
\[
\tau := \left( K + \sigma^{-1} \right) \left( \| u \|_\kappa + C(\Omega) \| u^\flat \|_2^2 + 100 \right),
\]
with $K$ and $\sigma$ coming from Proposition~\ref{Prop:Stage} below. We put $\hat{A} := \delta_1 \tau^{-1} A$, where $\delta_1$ is defined in~\eqref{eq:stale-do-iteracji} and notice that if we find a solution $ \left( v, w, \kerk \right)$ to 
\[
\cchermp{v}{w} + \kerk  = \hat{A}
\]
then $\left( \delta_1^{-1/2}\tau^{1/2}v, \delta_1^{-1}\tau w, \delta_1^{-1} \tau \kerk \right)$ is the solution to
\[
\cchermp{v}{w} + \kerk = A
\]
and thus solves~\eqref{eq:MA-vw}. However for such a choice of $\hat{A}$ the triple $\left( \delta_1^{1/2} \tau^{-1/2} u^\flat, 0, 0 \right)$ satisfies the assumptions of Proposition~\autoref{Prop:Stage} at step $q=0$ with $A$ replaced by $\hat{A}$.

\subsection{Stage}

Here we will construct the stage iteration that will be crucial for the regularity of the final solution. 

At this point we may define the amplitudes $\delta$ and frequencies $\lambda$ as
\begin{equation}\label{eq:stale-do-iteracji}
	\delta_q := a^{-b^q} \qquad \lambda_q := a^{cb^{q+1}}. 
\end{equation}
for some positive constants $a, b, c > 1$ to be determined later. We also define the deficit tensor as
\[
D_q := A \cchermm{v_q}{w_q} - \kerk_q -\delta_{q+1} \Id.
\]
$D_q + \delta_{q+1}\Id$ measures the failure of the data $(v_q,w_q, \kerk_q)$ at step $q$ to solve the equation. Let us also define $\alpha$ as a H\"{o}lder index between 0 and some fixed $\alpha_0 \leq \kappa$.

The construction of the proof will mostly follow~\cite{LiQiu}, however to apply Proposition \autoref{Prop:diagonalizacja} we need to pay the price of controlling the H\"{o}lder norm of the deficit tensor which brings additional estimates, more in line with~\cite{CaoSzek19}. 

\begin{Prop}[Stage]\label{Prop:Stage}
	There are positive constants $K > 1$ and $\sigma > 0$ for which the following holds: if function a $v_q$ and a map $w_q$ both of class $C^2$ as well as a matrix field $\kerk_q$ of class $C^1$ satisfy
	\begin{align}
		\| v_q \|_1 + \| w_q \|_1 + \| \kerk_q \|_0 &\leq K,  \label{eq:C1-est-zalozenie}
		\\
		\| v_q \|_2 + \| w_q \|_2 + \| \kerk_q \|_1  &\leq K \delta^{1/2}_q\lambda_q,\label{usetocite} \\
		\| D_q \|_\alpha &\leq \sigma\delta_{q+1},\label{eq:deficit-Holder-est}
	\end{align}
	then there is a function $v_{q+1}$, a map $w_{q+1}$, both of class $C^2$, and a matrix field $\kerk_{q+1}$ of class $C^1$ in the weak kernel of $\curl \, \overline{\curl}$ such that
	\begin{align}
		\| v_{q+1} - v_q \|_1 \leq K\delta_{q+1}^{1/2}, \quad \| w_{q+1} - w_q \|_1 &\leq K \delta^{1/2}_{q+1}, \label{eq:stage-krok-C1-est} 
		\\
		\| v_{q+1} \|_2 + \| w_{q+1} \|_2 + \| \kerk_{q+1} \|_1 &\leq K \delta^{1/2}_{q+1} \lambda_{q+1}, \label{eq:stage-krok-C2-est} \\
		\| D_{q+1} \|_\alpha &\leq \sigma \delta_{q+2}. 
		\label{eq:stage-blad-Holder-est}
	\end{align}
	
\end{Prop}

\begin{proof}
	First, we mollify $A, v_q, w_q$ and $\kerk_q$ on a scale $l$ to get $\tilde{A}, \tilde{v}, \tilde{w}$ and $\tilde{\kerk}$. Then we define the following matrix
	\[
	\tilde{D} := \tilde{A} \cchermm{\tilde
		{v}}{\tilde{w}} - \tilde{\kerk} - \delta_{q+2} \Id.
	\]
	For convenience we define the following constant
	\begin{equation}\label{eq:nu-def}
	 \nu := K \delta_q^{1/2} \delta_{q+1}^{-1/2} \lambda_q.
	\end{equation}
	Let us note that with such a definition we have the identity
	\begin{equation}\label{eq:est-C2-inaczej}
		K \delta_q^{1/2} \lambda_q = \nu \delta_{q+1}^{1/2}.
	\end{equation}

	By mollifiaction estimates in Lemma~\autoref{Lem:mollify-est} and (\ref{usetocite}) we have
    \begin{align}
		\begin{split}\label{eq:krok-Ck-mollified-roznica}
			|\tilde{v} - v_q |_j + |\tilde{w} - w_q|_j &\leq C\left( |v_q|_2 + |w_q|_2 \right) l^{2-j}, \quad \text{ for } j = 0,1,
			\\
			|\tilde{v}|_{2+j} + |\tilde{w}|_{2+j} &\leq CK \delta_q^{1/2} \lambda_q l^{-j}, \quad \text{ for } j \geq 0.
		\end{split}
    \end{align}
	The same Lemma together with~\eqref{eq:deficit-Holder-est} also imply
	\begin{align}
		\| \tilde{D} - \delta_{q+1}\Id \|_\alpha &\leq \| D_q * \phi_l \|_\alpha + 
		\frac{1}{2} \left\| \left( \del v_q \otimes \bar{\del} v_q \right) * \phi_l - \del \tilde{v}_q \otimes \bar{\del} \tilde{v}_q \right\|_\alpha 
		+ \delta_{q+2} \nonumber 
		\\ 
		&\leq \| D_q \|_\alpha + C l^{2-\alpha} \| v_q \|^2_2 + \delta_{q+2} \nonumber
		\\
		&\leq 2 \sigma \delta_{q+1} + C \nu^2 l^{2 - \alpha} \delta_{q+1}, \label{eq:mollified-deficit-est-Holder}
	\end{align}
	provided we have
	\begin{equation}\label{eq:delta-krok-est}
		\delta_{q+2} \leq \sigma \delta_{q+1}.
	\end{equation}
	
	Let us fix now mollification scale $l$ as
	\begin{equation}\label{eq:skala-l-def}
		l^{2 - \alpha} := \frac{\sigma}{C K^2} \delta_{q+1} \delta_q^{-1} \lambda_q ^{-2}.
	\end{equation}
	That way~\eqref{eq:nu-def},~\eqref{eq:mollified-deficit-est-Holder},~\eqref{eq:delta-krok-est} and~\eqref{eq:skala-l-def} give the following estimate
	\[
	\| \tilde{D} - \delta_{q+1}\Id \|_{\alpha} \leq 3 \sigma \delta_{q+1}
	\]
	and recalling the mollification estimates \eqref{uselateron} and \eqref{eq:est-mollified-Holder} we also get 
	\begin{equation}\label{eq:deficit-k-Holder-est}
		| \tilde{D} - \delta_{q+1} \Id |_{k;\alpha} \leq C_k \sigma \delta_{q+1} l^{-k}.
	\end{equation}
	Provided we choose $\sigma < \tilde{\sigma}/3$, where $\tilde{\sigma}$ is the constant appearing in Proposition~\autoref{Prop:diagonalizacja} we may apply this proposition to $\tilde{D}/ \delta_{q+1}$ to get an element $\hat{\kerk} \in \ker_w\left( \curl\overline{\curl} \right)$ and functions $ \hat{d}_1, \ldots, \hat{d}_n $ such that
	\begin{equation}\label{eq:kernel-trick}
		\tilde{D} = \delta_{q+1} \hat{\kerk} + \delta_{q+1} \mathrm{diag} \left( \hat{d}^2_1,\ldots, \hat{d}^2_n \right).
	\end{equation}
	Recall that by~\eqref{eq:diago-est} and~\eqref{eq:deficit-k-Holder-est} we have
	\begin{align}
	    \begin{split}\label{eq:diago-delta-est}
			\sum_j \delta_{q+1} \| \hat{d}_j - 1 \|_\alpha + \delta_{q+1} \| \hat{\kerk} \|_\alpha &\leq C \sigma \delta_{q+1},
			\\
			\sum_j \delta_{q+1} \| \hat{d}_j \|_{k ; \alpha} + \delta_{q+1} \| \hat{\kerk} \|_{k ; \alpha} \leq C \| \tilde{D} - \delta_{q+1}\Id \|_{k ; \alpha} &\leq C \sigma \delta_{q+1} l^{-k}.
       \end{split}
	\end{align}
	At this point we are able to define the claimed element $\kerk_{q+1}$ of the weak kernel as
	\begin{equation}
		\kerk_{q+1} := \delta_{q+1} \hat{\kerk}.
	\end{equation}
	and amplitudes as
	\[
	d_j := \delta_{q+1}^{1/2} \hat{d}_j,
	\]
	which leads us to
	\begin{align}
		\| d_j \|_0 \leq (C \sigma+1) \delta^{1/2}_{q+1}, &\quad | d_j |_k \leq C \delta^{1/2}_{q+1} l^{-k}, \label{eq:dj-Ck-est}
		\\
		\| \kerk_{q+1} \|_0 \leq  C \sigma \delta_{q+1}, &\quad | \kerk_{q+1} |_1 \leq C \delta_{q+1} l^{-1}. \label{eq:kerk-Ck-est}
	\end{align}
	
	The final perturbation scheme is the following. Let us put $\hat{v}_0 = \tilde{v}$ and $\hat{w}_0 = \tilde{w}$. Then
	\begin{align*}
		\hat{v}_j =& \hat{v}_{j-1} + \frac{1}{\mu_j} \Gamma_1 \left( d_j(z), \mu_j \left( z + \bar{z} \right) \cdot e_j \right), \\
		\hat{w}_j =& \hat{w}_{j-1} - \frac{1}{\mu_j} \Gamma_1 \left( d_j(z), \mu_j ( z + \bar{z} ) \cdot e_j \right) \bar{\del} \hat{v}_{j-1} + \frac{1}{\mu_j} \Gamma_2 \left( d_j(z), \mu_j ( z + \bar{z} ) \cdot e_j \right) e_j,
	\end{align*}
	for $j = 1, \ldots, n$. Here $e_j$ are the vectors of the canonical basis of $\C^n$ and $\{ \mu_j \}$ is an increasing sequence of large constants to be fixed later on, with the exception of $\mu_n$, which we set as
	\[
	\mu_n := \lambda_{q+1}.
	\]
	We also introduce additional constant $\mu_0$ and assume $l^{-1} \leq \mu_0 \leq \mu_1 \leq \ldots \leq \mu_n$. Exploiting~\eqref{eq:perturbacje-est} we get the estimates
	\begin{align*}
		\| \hat{v}_{j+1} - \hat{v}_{j} \|_0 &\leq C \frac{\| d_j \|_0}{\mu_{j+1}}, \\
		| \hat{v}_{j+1} - \hat{v}_{j} |_1 &\leq C \left( \frac{ | d_{j+1} |_1 }{ \mu_{j+1} } + | d_{j+1}|_0 \right), \\
		| \hat{v}_{j+1} - \hat{v}_{j} |_2 &\leq C \left( \frac{| d_{j+1} |_2}{ \mu_{j+1} } + | d_{j+1} |_1 + \mu_{j+1}| d_{j+1} |_0 \right), \\
		| \hat{v}_{j+1} - \hat{v}_{j} |_3 &\leq C \left( \frac{ |d_{j+1}|_3 }{\mu_{j+1}} + |d_{j+1}|_2 + \mu_{j+1}|d_{j+1}|_1 +  \mu_{j+1}^2|d_{j+1}|_0\right).
	\end{align*}
	Putting~\eqref{eq:dj-Ck-est} into the above we get that
	\begin{align}\label{eq:vj-Ck-krok-roznica}
     | \hat{v}_{j+1} - \hat{v}_j |_k \leq C(\sigma) \delta^{1/2}_{q+1} \mu_{j+1}^{k-1},\ k\in\lbrace0,1,2,3\rbrace.
	\end{align}
	Moreover, from~\eqref{eq:C1-est-zalozenie},~\eqref{eq:est-C2-inaczej}, ~\eqref{eq:krok-Ck-mollified-roznica} and the above we get the following estimates for $| \hat{v}_j |_k$ with $k \in \{1,2,3 \}$:
	\begin{align}
		\begin{split}
			| \hat{v}_j |_1 \leq | \hat{v}_0 |_1 + C \delta^{1/2}_{q+1} &\leq C K,
			\\
			| \hat{v}_j |_k \leq | \hat{v}_0 |_k + C \delta^{1/2}_{q+1} \mu_{j}^{k-1} &\leq C \delta^{1/2}_{q+1} \mu_{j}^{k-1}, \quad k = 2,3 \label{eq:v_j-Ck-est}
		\end{split}
	\end{align}
    provided
	\begin{equation}\label{eq:est-oczywisty}
		\lambda_q \leq \mu_0 \leq \mu_j \ \ \text{ and } \ \ \nu \leq l^{-1}
	\end{equation}
	The second assumption is always satisfied provied the base $a$ in~\eqref{eq:stale-do-iteracji} is big enough.
	
	Similarly, for the maps $\hat{w}_j$ we get the estimates
	\begin{align*}
		\| \hat{w}_{i+1} - \hat{w}_{i} \|_0 &\leq \frac{C}{\mu_{i+1}}, \\
		| \hat{w}_{i+1} - \hat{w}_{i} |_1 &\leq C \left( \frac{ | d_{i+1} |_1 | \hat{v}_{i} |_1}{ \mu_{i+1} } +
		\frac{ | d_{i+1} |_0 |\hat{v}_{i} |_2 }{ \mu_{i+1} } +
		| d_{i+1} |_0| \hat{v}_{i} |_1 + 
		\frac{ |d_{i+1} |_1 | d_{i+1} |_0 }{ \mu_{i+1} } + 
		| d_{i+1} |^2_0 \right), \\
		| \hat{w}_{i+1} - \hat{w}_{i} |_2 &\leq C \left( \frac{1}{ \mu_{i+1} } \left( { |d_{i+1}|_2 | \hat{v}_{i}|_1}  +
		{ | d_{i+1} |_1 | \hat{v}_{i} |_2 } +
		{ |d_{i+1}|_0 |\hat{v}_{i}|_3 } +
		{ |d_{i+1}|_2 |d_{i+1}|_0 } + 
		{ |d_{i+1}|^2_1 } \right) \right.\\
		&\phantom{ C \ ( \frac{1}{ \mu_{i+1} } ( \ } +| d_{i+1} |_1 | \hat{v}_{i} |_1 + 
		| d_{i+1} |_0 | \hat{v}_{i} |_2 +
		| d_{i+1} |_0 | d_{i+1} |_1 \\
		&\phantom{ C \ ( \frac{1}{ \mu_{i+1} } ( \ } \left. +\mu_{i+1} \left( | d_{i+1} |^2_0 + | d_{i+1} |_0 | \hat{v}_{i} |_1 \right) \right). 
	\end{align*}
(keep in mind that $\Gamma_1$ is linear in the first variable).

	Again, we can plug into it~\eqref{eq:v_j-Ck-est} and~\eqref{eq:dj-Ck-est} to get
	\begin{align}
		\begin{split}\label{eq:w_j-Ck-est}
			| \hat{w}_{j+1} - \hat{w}_j |_1 &\leq C \left( \delta_{q+1} + \delta_{q+1}^{1/2} \right);
			\\
			| \hat{w}_{j+1} - \hat{w}_j |_2 &\leq C \left( \mu_{j+1} \delta_{q+1}^{1/2} \right).
		\end{split}
	\end{align}
	
	At this point we can notice that~\eqref{eq:vj-Ck-krok-roznica},~\eqref{eq:w_j-Ck-est} and~\eqref{eq:kerk-Ck-est} imply~\eqref{eq:stage-krok-C1-est} and~\eqref{eq:stage-krok-C2-est} for $\hat{v}_n, \hat{w}_n, \kerk_{q+1}$ provided the constant $K$ from the claim is taken big enough.
	
	Finally, we may estimate the new deficit tensor. At each step we get a perturbation error defined as
	\begin{align*}
		\mathcal{E}_j := &\frac{1}{2} \del v_{j} \otimes \bar{\del} v_{j} + \herm \del w_j - \frac{1}{2} \del v_{j-1} \otimes \bar{\del} v_{j-1} - \herm \del w_{j-1} - d_j^2 \left( e_j \otimes {e}_j \right)= 
		\\
		- &\frac{1}{ \mu_j } \Gamma_1 \del \bar{\del} \hat{v}_{j-1} + \frac{1}{ \mu_j } \left( \del_s \Gamma_2 + \del_s \Gamma_1 \del_t \Gamma_1 \right) \herm \left( \del d_j \otimes e_j \right)
		+ \frac{1}{ 2 \mu^2_j } | \del_s \Gamma_1 |^2 \del d_j \otimes \bar{\del} d_j.
	\end{align*}
	Using~\eqref{eq:perturbacje-est},~\eqref{eq:dj-Ck-est} and~\eqref{eq:v_j-Ck-est} we estimate this expression by
	\begin{align}
		\begin{split}
			\| \mathcal{E}_j \|_0 \leq C &\left( \frac{1}{\mu_j} \| \Gamma_1 \|_0 | \hat{v}_{j-1} |_2 + \frac{1}{ \mu_j } \left( \| \del_s \Gamma_2 \|_0 + \| \del_s \Gamma_2 \del_t \Gamma_1 \|_0 \right) | \del d_j |_0 \right. \\
			& \left. \phantom{((}+ \frac{1}{ \mu_j^2  } \| \del_s \Gamma_1 \|^2_0 | \del d_j |^2 \right)
		\end{split} \nonumber
		\\
		\leq \ & \delta_{q+1} \frac{\mu_{j-1}}{\mu_j}. \label{eq:perturbation-C0-est}
	\end{align}
	For the interpolation inequalities we also need the $C^1$-seminorm estimate of the deficits, thus (recall once again that $\Gamma_1$ is linear in $s$)
	\begin{align}
		| \mathcal{E}_j |_1 &\leq \frac{1}{\mu_j} \left( \|\del_s \Gamma_1 \|_0 |\hat{v}_{j-1}|_2 |d|_1 + \mu_j \| \del_t \Gamma_1 \|_0 | \hat{v}_{j-1} |_2 + \|\Gamma_1\|_0 |\hat{v}_{j-1}|_3 \right) \nonumber
        \\ 
		&+ \frac{|d|_1}{\mu_j} \left( |d|_1 \|\del_s \Gamma_1 \del^2_{st} \Gamma_1 \|_0 + \mu_j \left( \|\del_{st}^2 \Gamma_2 \|_0 + \|\del^2_{st} \Gamma_1 \del_t \Gamma_1 \|_0 + \|\del_s \Gamma_1 \del^2_{tt} \Gamma_1 \|_0 \right) \right) \nonumber
        \\
        &+ \frac{|d|_2}{\mu_j} \left( \|\del_s \Gamma_2 \|_0 + \|\del_s \Gamma_1 \del_t \Gamma_1 \|_0 + \mu^{-1}_j\|d_j\|_1 \|\del_s \Gamma_1 \|_0^2 \right) + \frac{|d_j|_1^2}{\mu_j} \nonumber
        \\
		& \leq \frac{C \delta_{q+1}}{\mu_j} \left( \mu_{j-1} \left( l^{-1} + \mu_j + \mu_{j-1} \right) + l^{-2} + \mu_j + l^{-1} \right) \nonumber
		\\
		&\leq C \delta_{q+1} \mu_{j-1} \label{eq:perturbation-C1-est}
	\end{align}
	and the last estimate holds since $l^{-1} \leq \mu_j$.
	
	Since the new deficit tensor is defined as
	\[
	D_{q+1} = A \cchermm{v_{q+1}}{w_{q+1}} - \kerk_{q+1} - \delta_{q+2} \Id,
	\]
	thanks to~\eqref{eq:kernel-trick} we arrive at the estimate
	\[
		\| D_{q+1} \|_\alpha \leq \| A - \tilde{A} \|_\alpha + \sum_{j=1}^n \| \mathcal{E}_j \|_\alpha .
	\]
	We will get~\eqref{eq:stage-blad-Holder-est} by estimating each H\"{o}lder norm in the expression above. Thus using the H\"{o}lder interpolation with~\eqref{eq:perturbation-C0-est} and~\eqref{eq:perturbation-C1-est} we get
	\[
		\| \mathcal E_j \|_\alpha \leq C \| \mathcal E_j \|_0^{1 - \alpha} | \mathcal E_j |_1^{\alpha} \leq C \delta_{q+1} \frac{\mu_{j-1}}{\mu_j^{1-\alpha}}
	\]
	and
	\[
	\| A - \tilde{A} \|_\alpha \leq C l^{\kappa - \alpha} \| A \|_\kappa,
	\]
	which brings the H\"{o}lder norm estimate of $D_{q+1}$ to
	\[
	\| D_{q+1} \|_\alpha \leq C \left( l^{\kappa - \alpha} \| A \|_\kappa + \delta_{q+1} \sum_j \frac{ \mu_{ j -1 } }{ \mu_j^{1 - \alpha} } \right)
	\]
	thus to get the claimed estimate it suffices that
	\begin{equation}\label{eq:stale-nierownosci-3}
		l^{\kappa - \alpha} \leq \frac{\sigma \delta_{q+2}}{C (n+1)} 
		\quad \text{ and } \quad 
		\frac{ \mu_{j-1} }{ \mu_j^{1- \alpha} } \leq \frac{\sigma \delta_{q+2}}{ (n+1) C \delta_{q+1}}.
	\end{equation}
	Thus assuming~\eqref{eq:stale-nierownosci-3} we obtain
	\[
	\| D_{q+1} \|_\alpha \leq \sigma \delta_{q+2},
	\]
	which concludes the proof of the result.
\end{proof}

{Step 3. {\bf Conclusion}.} In order to finish the proof of the main theorem we must show that $v_q$ and $w_q$ converge in an appropriate $C^{1,\beta}$ norm, $\kerk_q$ converges in $C^\beta$ and $D_q \rightarrow 0$ in $C^0 \left( \Omega, \C^{n \times n}_\herm \right)$.

For the convergence of $v_q$ and $w_q$ we use the interpolation inequalities of H\"{o}lder norms. Thus from~\eqref{eq:stage-krok-C1-est} and~\eqref{eq:stage-krok-C2-est} we have
\begin{equation}
	\| v_{q+1} - v_q \|_{1;\beta} \leq \| v_{q+1} - v_q \|_{1}^{1 - \beta} \| v_{q+1} - v_q \|_{2}^{\beta} \leq K \delta_{q+1}^{1/2} \lambda^\beta_{q+1},
\end{equation}
which also works for $w_q$. Thus 
\[
\| v_{q+1} - v_q \|_{1;\beta} \leq K a^{b^q \left( bc\beta - \frac{1}{2} \right)}
\]
and it becomes a Cauchy sequence as long as
\begin{equation}\label{eq:holder-wykl-1}
	\beta < \frac{1}{2bc}.
\end{equation}

We must also justify inequalities~\eqref{eq:stale-nierownosci-3}. At this point me may fix the constants $\mu_j$. Let us put
\begin{equation}\label{eq:mu_0-definicja}
	\mu_0 := l^{-1}
\end{equation}
and suppose for each $j = 1,\ldots, n-1$ we have
\begin{equation}
	\frac{ \mu_{j-1} }{ \mu_j^{1-\alpha}} = \frac{ \sigma \delta_{q+2} }{ C' \delta_{q+1}}
\end{equation}
for some constant $C' \geq (n+1)C$ with $C$ taken from~\eqref{eq:stale-nierownosci-3}. Let us note that it will give an increasing sequence of constants for base $a$ big enough. Moreover, then by the defintion of $l$, the condition~\eqref{eq:est-oczywisty} is easily satisfied. Therefore for $\lambda_{q+1}$ (= $\mu_n$) we have the following condition
\begin{equation}
	\lambda_{q+1} \geq \mu_0^{(1-\alpha)^{-n}} \delta_{q+1}^{ \alpha^{-1} \left( (1-\alpha)^{-n} - 1)\right)} \delta_{q+2}^{ \alpha^{-1} \left( 1 - (1-\alpha)^{-n} \right)} C( \sigma, \alpha).
\end{equation}
After plugging in~\eqref{eq:mu_0-definicja} and taking $a$ big enough so that we may ignore all the fixed constants we arrive at the following inequalities
\begin{align}
	cb \left[ b(2-\alpha)(1 - \alpha)^{n} - 2 \right] &\geq (b-1)\left[ 1 + b(2 - \alpha) \left( \frac{1 - (1 - \alpha)^{n}}{ \alpha } \right) \right] \label{eq:stale-nierownosc-1},\\
	\kappa \left( 2bc + b - 1 \right) &\geq b^2(2-\alpha) + \alpha (2bc + b - 1) \label{eq:stale-nierownosc-2}.
\end{align}
Thus for $b$ we get 
\[
b > \frac{2}{ (2 - \alpha) (1 - \alpha)^n } \xrightarrow[\alpha \rightarrow 0]{} 1,
\]
which brings the following estimate for the constant $c$ as $\alpha \rightarrow 0$
\begin{equation}
	c \geq \frac{ (b-1)(1 + 2nb) }{ 2b (b-1)} \rightarrow n + \frac{1}{2}\ \ {\rm as}\ \ b\rightarrow 1^+.
\end{equation}
That way~\eqref{eq:holder-wykl-1} is satisfied whenever
\begin{equation}\label{eq:holder-wykl-2}
	\beta < \frac{1}{1+ 2n}
\end{equation}
for $b$ arbitrarily close to $1$.

Similar reasoning for $\kappa$ gives us estimate
\[
 \kappa > \frac{2}{1 + 2n}.
\]
Recall that the integrability exponent $p$ of the right-hand side of the equation is related to $\kappa$ by identity $\kappa = 2 - \frac{2n}{p}$ which brings the limit value for $p$ at
\[
 p > n + \frac{1}{2}.
\]

There is also a question of convergence of the kernel elements $\kerk_q$ in $C^\beta$. However, this follows from the conditions that are already imposed, namely let $r \geq q$ then by H\"{o}lder interpolation and~\eqref{eq:kerk-Ck-est}
\[
\| \kerk_q - \kerk_r \|_\beta \leq 2 \| \kerk_q \|_\beta \leq C \delta_{q+1} l^{-\beta}.
\]
Therefore it converges whenever $v_q$ and $w_q$ do.

As a result $v=lim_{q\rightarrow\infty}v_q$ exists and is a solution to~\eqref{eq:MA-vw} because the deficit tensor vanishes in the limit (see Proposition ~\ref{Prop:Stage}).

The only thing left to do is make sure that the solution is close enough to the chosen function $u^\flat$. Recall that $v_0 = u^\flat$ was the starting point of the iteration scheme. By~\eqref{eq:stage-krok-C1-est} we have
\begin{align}
	\| u^\flat - v \|_0 \leq \| u^\flat - v_1 \|_0 + \sum_{q = 1}^{\infty} \| v_{q} - v_{q+1} \|_0 \leq \sum_{q=0}^{\infty} K \delta_{q+1}^{1/2}.
\end{align}
For a fixed $p$ and $b$ one can always take $a$ big enough so that the second term above is smaller than any pregiven $\varepsilon$.

\section{Real 2-Hessian}

Let us mention briefly that a modification of Proposition \ref{Prop:diagonalizacja} can be applied to the real 2-Hessian equation yielding very similar estimates.

\paragraph{{\bf Real kernel}.} For the real symmetric matrix field $A$ defined on a simply-connected domain $\Omega$ the kernel equation for $\tr( \curl_L \curl_R)$ is almost the same. Namely the coefficient of the field $A$ must satisfy
\[
\sum_j \Delta_j a_{jj} = 2 \sum_{i < j} \del_{ij}^2 a_{ij}.
\]

\begin{Prop}\label{Prop:diagonalizacja-R}
	For any $j \in \N , \ 0 < \alpha < 1$ there exist $M_1, M_2,\ldots$ and $\sigma_1$ depending on $j, \alpha$ and $n$ such that the following holds. If $S \in C^{j,\alpha} \left( \Omega, \R^{n \times n}_{\mathrm{sym}} \right)$ satisfies
	\[
	\| S - \Id \|_\alpha \leq \sigma_1,
	\]
	then there exist $\kerk \in C^{j+1, \alpha} \left( \Omega, \R^{n \times n}_{\mathrm{sym}} \right) \, \cap \, \ker_w \left( \mathrm{curl} \, \mathrm{curl} \right)$ and $d_1,\ldots,d_n \in C^{j, \alpha} \left( \Omega, \R \right)$ such that
	\begin{equation}\label{eq:diag-potencjal-R}
		S + \kerk = diag(d^2_1,\ldots,d^2_n)
	\end{equation}
	and the following estimates hold:
	\begin{align}
		\begin{split}\label{eq:diago-est-R}
			\| d - 1 \|_\alpha + \| \kerk \|_\alpha &\leq M_1 \| S - \Id \|_\alpha; \\
			\| d \|_{j ; \alpha} + \| \kerk \|_{j ; \alpha} &\leq M_j \| S - \Id \|_{j ; \alpha}.
		\end{split}
	\end{align}
\end{Prop}
\begin{proof}
	We proceed in the same manner as with the complex conterpart. However, because the Hodge star behaves differently for real and complex spaces the final kernel field will have a different form. Let us denote by $s^{ij}$ the entries of $S$ and by $\kerk^{ij}$ the coefficients of the matrix field $\kerk$ from the kernel. Let us write down down the weak form of the kernel equation. Thus for any test function $\varphi$ we must have
	\begin{align*}
		0 &= \int_\Omega \varphi \, \mathrm{curl curl}\kerk \\
		&= \int_\Omega \varphi \left( \sum_{i} \Delta_i \kerk^{ii} - 2 \sum_{i < j} \del^2_{ij} \kerk^{ij} \right) \\
		&= - \int_\Omega \sum_i \del_i \varphi \left( \sum_{j \neq i } \del_i \kerk^{jj}- \del_j \kerk^{ij} \right).
	\end{align*}
	This will vanish if we can find a $(n-2)$-form $F$ such that the last integral can be expressed as
	\[
	\int_\Omega d\varphi \wedge d F.
	\]
	Let us denote by $dx_I^i$ the $(n-1)$-form $dx_1 \wedge \ldots dx_{i-1} \wedge dx_{i+1} \wedge \ldots \wedge dx_n$. We define the $(n-2)$-form $dx_I^{ij}$ in a similar way for $i<j$. Thus if we write down the form $dF$ as $\sum_i F_i\, dx_I^i$ we get
	\[
	F_i = (-1)^{i-1} \sum_{j \neq i} (\del_i \kerk^{jj} - \del_j \kerk^{ij})
	\]
	Let us further assume that $F = *G$ for some 2-form $G$ and $G = dH$ for a 1-form $H$. Let us denote the coefficients of $H$ by $h^i$, i.e. $H = \sum_i h^idx_i$. Calculating the dependence of $F$ on $H$ we arrive at
	\begin{align*}
		d*dH =& \\
		=&  d* \sum_{i < j} (\del_i h^j - \del_j h^i) dx_i \wedge dx_j \\
		=& d \sum_{i < j} (-1)^{i + j + 1}( \del_i h^j - \del_j h^i) dx_I^{ij} \\
		=& \sum_{i < j} (-1)^j( \del^2_{ii} h^j - \del^2_{ji} h^{i} )dx_I^j + (-1)^{i-1} (\del^2_{ij} h^{j} - \del^2_{jj} h^i) dx^i_I \\
		=& \sum_i (-1)^i \sum_{j \neq i} (\del^2_{jj} h^i - \del^2_{ji} h^j) dx_I^i.
	\end{align*}
	If we compare the result with the coefficients of $F_i$, we get
	\begin{equation}\label{eq:potencjaly-jadro}
		\sum_{j \neq i} \del_j \kerk^{ij} - \del_i \kerk^{jj} = \sum_{j \neq i} \del^2_{jj} h^i - \del^2_{ij} h^j.
	\end{equation}
	Suppose there are functions $\psi^{ij}$ for any $1 \leq i < j \leq n$ that satisfy
	\[
	h^i = \sum_{j < i} \del_j \psi^{ji} + \sum_{j > i} \del_j \psi^{ij}.
	\]
	If we plug this into~\eqref{eq:potencjaly-jadro} and compare the left and the right-hand side of we get 
	\begin{align*}
		&\sum_{j \neq i} \del_j \kerk^{ij} - \del_i \kerk^{jj} = \sum_{j \neq i} \left( \sum_{k < i} \del^3_{kjj}\psi^{ki} + \sum_{k > i} \del^3_{kjj} \psi^{ik} - \sum_{l < j} \del^3_{lij} \psi^{lj} - \sum_{l > j} 
		\del^3_{lij} \psi^{jl} \right)  \\
	&	= \sum_{j < i} \Delta \del_j \psi^{ji} + \sum_{j > i} \Delta \del_j \psi^{ij}  - 2\sum_{j < i} \del^3_{iij} \psi^{ij}-2\sum_{j > i} \del^3_{iij} \psi^{ji} - 2\sum_{j \neq i}\sum_{l\neq i,l < j}
		\del^3_{ilj} \psi^{lj}\\
		&- 2\sum_{j \neq i}\sum_{l\neq i,l > j}
		\del^3_{ilj} \psi^{jl}.
	\end{align*}
	If we now make a following ansatz that each $\psi_{ij}$ solves $\kerk^{ij} = \Delta \psi^{ij}$ the formula will hold provided we put
	\begin{equation}\label{diagterms}
		\kerk^{11} = \kerk^{22} = \ldots = \kerk^{nn} = -\frac{2}{n-1}\sum_{i<j} \del_{ij}^2\psi^{ij}.
	\end{equation}
	
	Going back to the field $S$ we solve $ \Delta \psi^{ij}= -s^{ij}$ for $i < j$ (with zero boundary) and then the elements on the diagonal by the formula (\ref{diagterms}). The claim follows again from the Schauder estimates applied to functions $\psi^{ij}$.
\end{proof}

\bibliographystyle{plain}
\bibliography{bibliografia_cvx_int}

\end{document}